\documentclass{amsart}
\usepackage{enumerate}
\usepackage[dvips]{graphicx}
\usepackage{color}
\usepackage{pstricks}
\usepackage{pst-node}
\newtheorem{theo}{Theorem}[section]
\newtheorem{cor}[theo]{Corollary}
\newtheorem{lem}[theo]{Lemma}

\newtheorem{prop}[theo]{Proposition}

\newtheorem{defn}[theo]{Definition}

\newtheorem{Notation}[theo]{Notation} 
\theoremstyle{definition}

\newtheorem*{notaetoile}{Notation}

\def\a{{\a}}
\def\a{{\mathfrak a}}

\def\Card{\operatorname{Card}}

\def\N{\mathbb N}
\def\P{\mathbb{P}}

\def\R{\mathbb R}

\def\Z{\mathbb{Z}}
\begin{document} 

\title[An interacting particle model and A Pieri-type formula for $SO(d)$]{An interacting particle model\\ and \\ a Pieri-type formula for the orthogonal group}

\author{Manon Defosseux}
\address{Laboratoire de Math\'ematiques Appliqu\'ees \`a Paris 5, Universit\'e Paris 5, 45 rue des  Saints P\`eres, 75270 Paris Cedex 06.}
\email{manon.defosseux@parisdescartes.fr}

\begin{abstract} We introduce a new interacting particle model with  blocking and pushing interactions.   Particles evolve on $\Z_+$ jumping on their own volition rightwards or leftwards  according to geometric jumps with parameter $q\in(0,1)$. We show that the model involves a Pieri-type formula for  the orthogonal group. We prove that the two extreme cases - $q=0$ and $q=1$ -   lead respectively  to a random tiling model studied in \cite{BorodinKuan} and  a random matrix model considered in \cite{DefosseuxLUE}. 
\end{abstract}
\maketitle 
\section{introduction}
In \cite{BorodinKuan} A. Borodin  and J. Kuan consider a random tiling model with a wall which is related to the Plancherel 
measure for the orthogonal group  and thus to representation theory of this group. Similar connection holds for  the interacting particle model and the random matrix model  considered in \cite{DefosseuxLUE}.  The aim of this paper is to establish a direct link between the random tiling model on one side  and the interacting particle model or the random matrix model on the other side. For this we consider an interacting particle model depending on a parameter and show that these models correspond to different parameter values. 

The paper is organized as follows. Definition of the set of Gelfand-Tsetlin patterns for the orthogonal group is recalled in section \ref{definitions}. Section \ref{modeldescription} is devoted to the description of the particle model. We recall in section \ref{Poisson} the description of an interacting particle model equivalent to the random tiling model studied in \cite{BorodinKuan}. Models considered in that paper involve Markov kernels which can be obtained with the help of a Pieri-type formula for the orthogonal group. These Markov kernels are constructed in section \ref{representation} after recalling some elements of representation theory. We describe the matrix model  related to the particle model  in section \ref{matrices}. Results are stated in section \ref{results} and proved in section \ref{proofs}.

\textit{Acknowlegments:}  The author would like to thank Alexei Borodin for his suggestions and helpful explanations.
\section{Gelfand-Tsetlin patterns}\label{definitions}
Let  $n$ be a positive integer. For $x, y\in \mathbb{R}^{n}$ such that $x_n\le\dots\le x_1$ and $y_n\le\dots\le y_1$, we write $x \preceq y$ if $x$ and $y$ are interlaced, i.e.\@
$$x_{n}\le y_n\le x_{n-1} \le \cdots\le x_1 \le y_1 .$$
When $x\in \mathbb{R}^{n}$ and $y\in\mathbb{R}^{n+1}$ we add the relation $y_{n+1} \leq x_{n}$. We  denote by $\vert x\vert$ the vector of $\R^n$ whose components are  the absolute values of those of $x$. 
\begin{defn}  \label{defnGT} Let $k$ be a positive integer.
\begin{enumerate}  
\item We denote by $GT_{k}$  the set of Gelfand-Tsetlin patterns defined by  
\begin{eqnarray*}
GT_{k}&=&\{(x^{1},\cdots, x^{k}): x^{i}\in \mathbb{N}^{j-1}\times \mathbb{Z}\mbox{ when } i=2j-1,\\ &&\hskip 15pt   \, x^{i}\in \mathbb{N}^{j} \mbox{ when } i=2j\mbox{, and } \vert x^{i-1} \vert \preceq \vert  x^{i} \vert, 1\leq i\leq k\}.
\end{eqnarray*}
  
 \item If $x=(x^{1},\dots,x^{k})$ is a Gelfand-Tsetlin pattern, $x^{i}$ is called the $i^{th}$ row of $x$ for $i\in\{1,\dots,k\}$.
  \item For $\lambda\in \Z^{[\frac{k+1}{2}]}$  the subset of Gelfand-Tsetlin patterns    having a $k^{th}$ row equal to $\lambda$ is denoted by $GT_k(\lambda)$ and its cardinality is denoted by $s_k(\lambda)$.   
 \end{enumerate}
\end{defn}  
 
Usually,  a Gelfand Tsetlin pattern   is represented by a triangular array as indicated at figure \ref{coneB} for $k=2r$.

\begin{figure}[h!]
\begin{pspicture}(2,4)(9,0.25)
 \put(1,0){$-x^{2r}_{1}$}  \put(2.5,0){$\cdots $}\put(4.5,0){$-x^{2r}_{r}$}   \put(5.75,0){$x^{2r}_{r}$}  \put(8,0){$\cdots $}  \put(9.5,0){$x^{2r}_{1}$}

 \put(1.75,0.5){$-x^{2r-1}_{1}$}  \put(2.75,0.5){$\cdots $}\put(3.75,0.5){$-x^{2r-1}_{r-1}$}  \put(5.25,0.5){$x^{2r-1}_{r}$}   \put(6.5,0.5){$x^{2r-1}_{r-1}$}  \put(7.75,0.5){$\cdots $}  \put(8.75,0.5){$x^{2r-1}_{1}$}
 
 \put(4.5,1){$\cdots$}  \put(6,1){$\cdots$} 
     \put(3,1.5){$-x_{1}^{4}$}  \put(4.5,1.5){$-x_{2}^{4}$} \put(5.75,1.5){$x_{2}^{4}$} \put(7.125,1.5){$x_{1}^{4}$} 
       \put(3.75,2){$-x^{3}_{1}$}       \put(5.25,2){$x^{3}_{2}$}   \put(6.5,2){$x^{3}_{1}$}    
      \put(4.5,2.5){$-x^{2}_{1}$}       \put(5.75,2.5){$x^{2}_{1}$} 
    \put(5.25,3){$x^{1}_{1}$}    

\end{pspicture}
  \caption{A Gelfand--Tsetlin pattern of   $GT_{2r}$}
  \label{coneB}
\end{figure}

\section{An interacting particle model with exponential jumps}\label{modeldescription}
In this section we  construct a random process $(X(t))_{t\ge 0}$ evolving on the subset of $GT_k$ of Gelfand-Tsetlin patterns with non negative valued components. This process can be viewed as an interacting particle model. For this, we associate to a Gelfand-Tsetlin pattern $x=(x^{1},\dots,x^{k})$, a configuration of particles on the integer lattice $\Z^2$ putting one particle labeled by $(i,j)$ at point $(k-i,x_j^{i})$ of $\Z^2$ for $i\in\{1,\dots,k\}$, $j\in\{1,\dots,[\frac{i+1}{2}]\}$. Several particles can be located at the same point. In the sequel we will say "particle $x_j^i$" instead of saying "particle labeled by $(i,j)$ located at point $(x_j^{i},k-i)$". Let $q\in(0,1)$. Consider two independent families  $$(\xi_j^i(n+\frac{1}{2}))_{i=1,\dots,k,j=1,\dots,[\frac{i+1}{2}];\, n\ge 0}, \quad \textrm{and } \quad (\xi_j^i(n))_{i=1,\dots k,j=1,\dots,[\frac{i+1}{2}];\, n\ge 1},$$   of identically distributed independent random variables  such that $$\P(\xi_1^1(\frac{1}{2})=x)=\P(\xi_1^1(1)=x)=q^x(1-q), \quad x\in \N,$$
and the Markov kernel $R$ on $\N$ defined by 
$$R(x,y)=\left\{
    \begin{array}{ll}
        \frac{1-q}{1+q}(q^{\vert x-y\vert}+q^{x+y}) & \mbox{if } y\in \N^* \\
        \\
        \frac{1-q}{1+q}q^x & \mbox{otherwise,}
    \end{array}
\right.$$
for $x\in \N$. Actually  the probability measure $R(x,.)$ on $\N$ is the law of the random variable $\vert x+\xi_1^1(1)-\xi_1^1(\frac{1}{2})\vert$.

Particles evolve as follows.  At time $0$ all particles are at zero, i.e. $X(0)=0$.    All particles, except those labeled by $(2l-1,l)$ for $l\in\{1,\dots,[\frac{k+1}{2}]\}$ (i.e. particles near the wall), try to jump to the left at times $n+\frac{1}{2}$ and to the right at times $n$, $n\in \N$.  Particles labeled by $(2l-1,l)$, for $l\in\{1,\dots,[\frac{k+1}{2}]\}$, jump on their own volition at integer times only. Notice that these particles can eventually move at half-integer times if they are pushed by another particle. Suppose that at time $n$ there is one particle at point $(k-i,X_j^{i}(n))$ of $\Z^2$, for $i=1,\dots,k$, $j=1,\dots,[\frac{i+1}{2}]$.  Positions of particles are updated downward as follows. Figure \ref{figuremodel} gives an example of an evolution of a pattern between times $n$ and $n+1$. In that example, the particles $(3,2)$ and $(4,2)$ are respectively pushed by the particles $(2,1)$ and $(3,1)$ and the particles $(3,1)$ and $(4,2)$ are respectively blocked by the particles $(2,1)$ and $(3,2)$ between times $n$ and $n+\frac{1}{2}$. The particle $(3,1)$ is pushed by  the particle $(2,1)$ and the particles $(3,2)$ and $(4,2)$ are respectively  blocked by the particles $(2,1)$ and $(3,1)$ between times $n+\frac{1}{2}$ and $n+1$.
\\

\noindent \underline{At time $n+1/2$} : 
All particles except particles $X^{2l-1}_{l}(n)$ for $l\in\{1,\dots,[\frac{k+1}{2}]\}$, try to jump to the left one after another in the lexicographic order pushing the   particles in order to stay in the set of Gelfand-Tsetlin patterns and being blocked by the initial configuration $X(n)$ of the particles. Let us indicate how the first three rows of a pattern are updated at time $n+\frac{1}{2}$.
\begin{itemize}
\item Particle $X_1^1(n)$ doesn't move. We let $$X_1^1(n+\frac{1}{2})=X_1^1(n).$$
\item Particle $X_1^{2}(n)$ tries to jump to the left according to a geometric jump. It is blocked by $  X_1^1(n)  $. If it is necessary it pushes $X_2^3(n)$ to an intermediate position denoted by $\tilde{X}_2^3(n)$, i.e.
\begin{align*}
X_1^{2}(n+\frac{1}{2})&=\max \big(  X_1^1(n)  ,X_1^2(n)-\xi_1^2(n+\frac{1}{2}) \big),\\
\tilde{X}_2^3(n)&=\min \big(X_2^3(n),X_1^2(n+\frac{1}{2}) \big).
\end{align*}
\item Particle $X_1^3(n)$ tries to move to the left according to a geometric jump being blocked by $X_1^2(n)$ :
\begin{align*} 
X_1^3(n+\frac{1}{2})=\max \big(X_1^2(n),X_1^3(n)-\xi_1^3(n+\frac{1}{2}) \big).
\end{align*}
 Particle $\tilde{X}_2^3(n)$ doesn't move. We let $$X_2^3(n+\frac{1}{2})=\tilde{X}_2^3(n).$$ \end{itemize}
Suppose now that rows 1 through $l-1$ have been updated for some $l>1$. Then the particles $X_{1}^l(n),\dots,X^l_{[\frac{l+1}{2}]}(n)$ of row $l$ are pushed to intermediate positions  
\begin{align*}
\tilde{X}_i^{l}(n)&=\min\big(X_i^{l}(n),X_{i-1}^{l-1}(n+\frac{1}{2})\big), \, i\in\{1,\dots,[\frac{l+1}{2}]\},
\end{align*}
whit the convention $X_{0}^{l-1}(n+\frac{1}{2})=+\infty$. Then particles $\tilde{X}_{1}^l(n),\dots, \tilde{X}^l_{[\frac{l}{2}]}(n)$  try to jump to the left according to geometric jump being blocked as follows by the initial position $X(n)$ of the particles. For $i=1,\dots,[\frac{l}{2}]$, 
\begin{align*}
X_i^{l}(n+\frac{1}{2})&=\max\big(X_i^{l-1}(n),\tilde{X}_i^{l}(n)-\xi_i^{l}(n+\frac{1}{2})\big).
\end{align*}
When $l$ is odd, particle $\tilde{X}^l_{\frac{l+1}{2}}(n)$ doesn't move and we let 
$$X_{\frac{l+1}{2}}^{l}(n+\frac{1}{2})=\tilde{X}_{\frac{l+1}{2}}^{l}(n).$$

\noindent \underline{At time $n+1$} : All particles except particles $X^{2l-1}_{l}(n+\frac{1}{2})$ for $l\in\{1,\dots,[\frac{k+1}{2}]\}$, try to jump to the right one after another in the lexicographic order pushing  particles in order to stay in the set of Gelfand-Tsetlin patterns and being blocked by the initial configuration $X(n+\frac{1}{2})$ of the particles.  Particles $X^{2l-1}_{l}(n+\frac{1}{2})$, for $l\in\{1,\dots,[\frac{k+1}{2}]\}$, try to move on their own volition according to the law $R(X_l^{2l-1}(n+\frac{1}{2}),.)$. The first three rows are updated as follows.
\begin{itemize}
\item Particle $X_1^1(n+\frac{1}{2})$ moves according to the law $R(X_1^1(n+\frac{1}{2}),.)$ pushing $X_1^2(n+\frac{1}{2})$ to an intermediate position $\tilde{X}_1^2(n+\frac{1}{2})$ : 
\begin{align*}  
X_1^1(n+1)&= \big\vert X_1^1(n+\frac{1}{2}) +\xi_{1}^1(n+1)-\xi_1^1(n+\frac{1}{2}) \big\vert, \\
\tilde{X}_1^2(n+\frac{1}{2})&=\max\big({X}_1^2(n+\frac{1}{2}),X_1^1(n+1)\big).
\end{align*}
\item Particle $\tilde{X}_1^2(n+\frac{1}{2})$  jumps to the right according to a geometric jump pushing $X_1^3(n+\frac{1}{2})$ to an intermediate position $\tilde{X}_1^3(n+\frac{1}{2})$, i.e.
\begin{align*}
X_1^{2}(n+1)&= \tilde{X}_1^2(n+\frac{1}{2}) + \xi_1^2(n+1),\\
\tilde{X}_1^3(n+\frac{1}{2})&=\max \big(X_1^3(n+\frac{1}{2}),X_1^2(n+1) \big).
\end{align*}
\item Particle $X_2^3(n+\frac{1}{2})$  tries to move according to the law $R(X_1^1(n+\frac{1}{2}),.)$. It is blocked by $X_1^2(n+\frac{1}{2})$. Particle $\tilde{X}_1^3(n+\frac{1}{2})$ moves to the right according to a geometric jump. That is
\begin{align*} 
X_2^3(n+1)&= \max(\big\vert X_2^3(n+\frac{1}{2}) +\xi_{2}^3(n+1)-\xi_2^3(n+\frac{1}{2}) \big\vert,X_1^2(n+\frac{1}{2})), \\
X_1^3(n+1)&=\tilde{X}_1^3(n+\frac{1}{2}) +\xi_1^3(n+1).
\end{align*} 
\end{itemize}
Suppose rows 1 through $l-1$ have been updated for some $l>1$. Then particles of row $l$ are pushed to intermediate positions  
\begin{align*}
\tilde{X}_i^{l}(n+\frac{1}{2})&=\max\big(X_{i}^{l-1}(n+1),X_i^{l}(n+\frac{1}{2})\big), \, i\in\{1,\dots,[\frac{l+1}{2}]\},
\end{align*}
with the convention $X_{\frac{l+1}{2}}^{l-1}(n+1)=0$ when $l$ is odd. Then  particles $\tilde{X}_1^l(n+ \frac{1}{2}),\dots,\tilde{X}_{[\frac{l}{2}]}^l(n+ \frac{1}{2})$ try to jump to the right according to geometric jump being blocked by the initial position of the particles as follows. For $i=1,\dots,[\frac{l}{2}]$,
\begin{align*}
X_i^{l}(n+1)&=\min\big(X_{i-1}^{l-1}(n+\frac{1}{2}),\tilde{X}_i^{l}(n+\frac{1}{2})+\xi_i^{l}(n+1)\big).
\end{align*}
When $l$ is odd, particle $X_{\frac{l+1}{2}}^{l}(n+\frac{1}{2})$ is updated as follows.
\begin{align*}
X_{\frac{l+1}{2}}^{l}(n+1)&= \min(\big\vert X_{\frac{l+1}{2}}^{l}(n+\frac{1}{2}) +\xi^{l}_{\frac{l+1}{2}}(n+1)-\xi^{l}_{\frac{l+1}{2}}(n+\frac{1}{2}) \big\vert, X_{\frac{l-1}{2}}^{l-1}(n+\frac{1}{2})).
\end{align*}

\begin{figure}[h!]
\begin{pspicture}(-9.5,6)(10,-5)
\put(-5,5.5){\underline{Interactions between times n and $n+\frac{1}{2}$}}
\psline[linecolor=gray]{->}(-8,-1)(2.5,-1)
\psline[linecolor=gray]{->}(-8,1)(2.5,1) 
\psline[linecolor=gray]{->}(-8,3)(2.5,3)
\psline[linecolor=gray]{->}(-8,-3)(2.5,-3)
\psline[linecolor=gray]{->}(-7,-3.3)(-7,4.5)

\psline[linecolor=lightgray]{<-}(0.5,-0.8)(0.5,0.1)
\psline[linecolor=lightgray]{<-}(-2,1.2)(-2,2.3)
\psline[linecolor=lightgray]{<-}(-4.9,3.1)(-4.9,4.3)
\psline[linecolor=lightgray]{<-}(-3,-0.7)(-3,0.1)
\psline[linecolor=lightgray]{<-}(-3.6,1.2)(-3.6,2.3)
\psline[linecolor=lightgray]{<-}(2,-2.7)(2,-2.1)
\psline[linecolor=lightgray]{->}(-2,-1.7)(-2,-1.2)
\psline[linecolor=lightgray]{->}(-3.6,-1.7)(-3.6,-1.2) 
\psline[linecolor=lightgray]{->}(1,-3.7)(1,-3.1) 
\psline[linecolor=lightgray]{->}(0,-3.7)(0,-3.1) 
\psline[linecolor=lightgray]{->}(-2,-3.7)(-2,-3.1) 
\psline[linecolor=lightgray]{->}(-3,-3.7)(-3,-3.1)

\put(-6.2,4.5){\scriptsize{$X^{1}_1(n+\frac{1}{2})=X^1_{1}(n)$}}
\put(1.7,-2){\scriptsize{$X^4_{1}(n)$}}
\put(-0.3,-4.1){\scriptsize{$X^4_{2}(n)$}}
\put(-2.3,2.6){\scriptsize{$X^2_{1}(n)$}}
\put(0.2,0.3){\scriptsize{$X^3_{1}(n)$}}
\put(-3.3,0.3){\scriptsize{$X^3_{2}(n)$}}
\put(-4.2,2.6){\scriptsize{$X^{2}_1(n+\frac{1}{2})$}}
\put(-2.3,-2){\scriptsize{$X^{3}_1(n+\frac{1}{2})$}}
\put(-4.5,-2){\scriptsize{$X_{2}^3(n+\frac{1}{2})$}}
\put(0.7,-4.1){\scriptsize{$X_{1}^4(n+\frac{1}{2})$}}
\put(-2.3,-4.1){\scriptsize{$\tilde{X}_{2}^4(n)$}}
\put(-3.8,-4.1){\scriptsize{${X}_{2}^4(n+\frac{1}{2})$}}

\psarc{->}(-2.8,0.8){0.8}{23}{159}
\psarc{-}(-1,-2){1.8}{37}{124}
\psarc{->}(-3.3,-0.9){0.3}{0}{180}
\psarc{->}(1.5,-3){0.5}{14}{168}
\psarc{->}(-1,-3.9){1.4}{48}{135}
\psarc{-}(-3,-3.9){1.4}{48}{90}

\psline{->}(-2,-0.5)(-2,-0.9) 
\psline{->}(-3,-2.5)(-3,-2.9) 

\psdots[dotstyle=o,dotsize=4pt](-3.6,1)(-2,1)
\psdots[dotstyle=square](0.5,-1)(-2,-1)
 \psdots(-4.9,3)
 \psdots[dotstyle=square*](-3.6,-1)(-3,-1)
 \psdots[dotstyle=diamond](2,-3)(1,-3)
 \psdots[dotstyle=diamond*](-2,-3)(0,-3)(-3,-3)

\end{pspicture} 
\begin{pspicture}(-9.5,5.5)(10,-4)

\put(-5,5.5){\underline{Interactions between times $n+\frac{1}{2}$ and $n+1$}}
\psline[linecolor=gray]{->}(-8,-1)(2.5,-1)
\psline[linecolor=gray]{->}(-8,1)(2.5,1) 
\psline[linecolor=gray]{->}(-8,3)(2.5,3)
\psline[linecolor=gray]{->}(-8,-3)(2.5,-3)
\psline[linecolor=gray]{->}(-7,-3.3)(-7,4.5)

\psline[linecolor=lightgray]{<-}(-3.6,1.2)(-3.6,2.3)
\psline[linecolor=lightgray]{<-}(-1.35,1.2)(-1.35,2.3)
\psline[linecolor=lightgray]{<-}(0.1,-0.8)(0.1,0.1)
\psline[linecolor=lightgray]{<-}(-1.35,-0.8)(-1.35,0.1)
\psline[linecolor=lightgray]{->}(1,-3.7)(1,-3.1) 
\psline[linecolor=lightgray]{->}(-2,-3.7)(-2,-3.1) 
\psline[linecolor=lightgray]{->}(-3,-3.7)(-3,-3.1) 
\psline[linecolor=lightgray]{<-}(-4.9,3.2)(-4.9,4.3)
\psline[linecolor=lightgray]{<-}(-5.77,3.2)(-5.77,4.3)
\psline[linecolor=lightgray]{<-}(2.2,-2.7)(2.2,-2.1)
\psline[linecolor=lightgray]{->}(-2,-1.7)(-2,-1.2)
\psline[linecolor=lightgray]{->}(-3.6,-1.7)(-3.6,-1.2)

\put(-5,4.5){\scriptsize{$X^{1}_1(n+\frac{1}{2})$}}
\put(-6.4,4.5){\scriptsize{$X^{1}_1(n+1)$}}
\put(-1.9,2.6){\scriptsize{$X^2_{1}(n+1)$}}
\put(-4.5,-2){\scriptsize{$X_{2}^3(n+1)$}}
\put(-2.3,-4.1){\scriptsize{${X}_{2}^4(n+1)$}}
\put(-0.2,0.3){\scriptsize{$X^3_{1}(n+1)$}} 
\put(-1.6,0.3){\scriptsize{$\tilde{X}^3_{1}(n+\frac{1}{2})$}} 

\put(1.9,-2){\scriptsize{$X^4_{1}(n+1)$}}
\put(-4.2,2.6){\scriptsize{$X^{2}_1(n+\frac{1}{2})$}}
\put(-2.3,-2){\scriptsize{$X^{3}_1(n+\frac{1}{2})$}}
\put(0.7,-4.1){\scriptsize{$X_{1}^4(n+\frac{1}{2})$}}
\put(-3.8,-4.1){\scriptsize{${X}_{2}^4(n+\frac{1}{2})$}}

\psarc{->}(-5.3,2.85){0.5}{34}{150}
\psarc{<-}(-2.5,0.6){1.2}{25}{153}
\psarc{<-}(-1.65,-1){0.35}{16}{159}
\psarc{<-}(-0.6,-1.1){0.7}{16}{163}
\psarc{-}(-3.2,-0.4){0.7}{125}{225}
\psarc{-}(-2,-3.9){1.4}{90}{132} 
\psarc{<-}(1.6,-3.7){1}{55}{128} 
 
\psline{->}(-3.6,0.17)(-3.6,-0.9) 
\psline{->}(-2,-2.48)(-2,-2.9) 

\psdots[dotstyle=o,dotsize=4pt](-3.6,1)(-1.35,1)
\psdots[dotstyle=square](-2,-1)(-1.35,-1)(0.1,-1)
 \psdots(-4.9,3)(-5.75,3)
 \psdots[dotstyle=square*](-3.6,-1)
 \psdots[dotstyle=diamond](1,-3)(2.2,-3)
 \psdots[dotstyle=diamond*](-3,-3)(-2,-3)

\end{pspicture} 

\caption{An example of blocking and pushing interactions between times $n$ and $n+1$ for $k=4$. Different kinds of dots represent different particles.}
\label{figuremodel}
\end{figure}
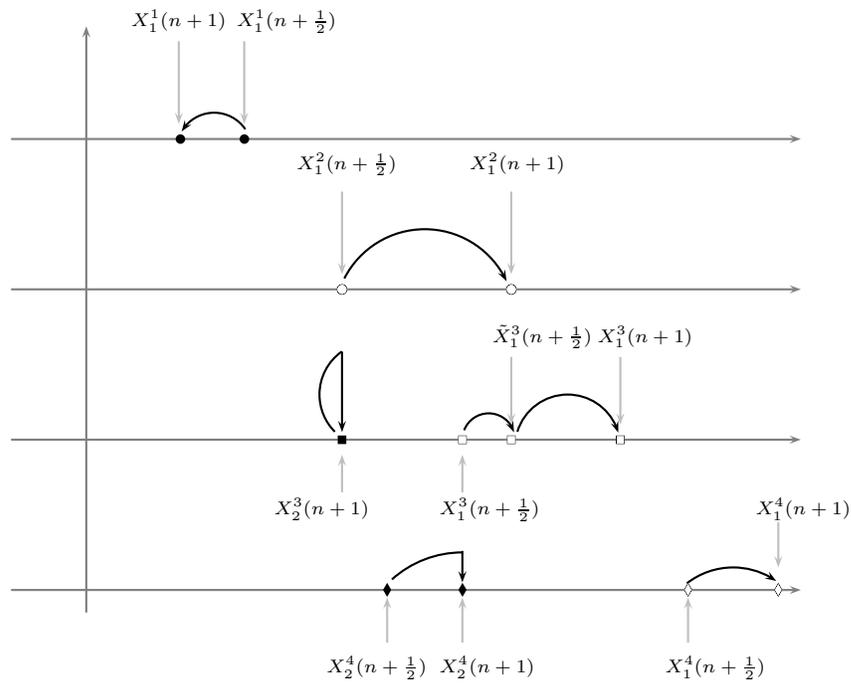 
\section{An interacting particle model with exponential waiting times}\label{Poisson}
In this section we describe an interacting particle model on $\Z^2$ where particles try to jump by one  rightwards or  leftwards  after   exponentially distributed waiting times. The evolution of the particles is described by a random process $(Y(t))_{t\ge 0}$ on the subset  of $GT_{k}$ of Gelfand-Tsetlin patterns with non negative valued components. As in the previous model, at time $t\ge 0$ there is one particle labeled by $(i,j)$ at   point $(k-i,Y^{i}_j(t))$ of the integer lattice, for $i=1,\dots,k$, $j=1,\dots,[\frac{i+1}{2}]$. Every particle tries to jump to the left or to the right by one after independent  exponentially distributed waiting time with mean $1$.  Particles are pushed and blocked according to the same rules as previously. That is when particle labeled by $(i,j)$ wants to jump to the right at time $t\ge 0$ then
\begin{enumerate}
\item if $i,j\ge 2$ and $Y_j^i(t^-)=Y_{j-1}^{i-1}(t^-)$ then the particles don't move and $Y(t)=Y(t^-)$,
\item else particles $(i,j),(i+1,j),\dots,(i+l,j)$ jump  to the right  by one for $l$  the largest integer such that $Y_j^{i+l}(t^-)=Y_j^i(t^-)$ i.e. $$Y_j^{i}(t)=Y_j^{i}(t^-)+1,\dots,Y_j^{i+l}(t)=Y_j^{i+l}(t^-)+1.$$ 
\end{enumerate}
When particle labeled by $(i,j)$ wants to jump to the left at time $t\ge 0$ then
\begin{enumerate}
\item if $i$ is odd, $j=(i+1)/2$ and $Y_j^{i}(t^-)=0$ then particle labeled by $(i,j)$ is reflected by $0$ and everything happens as described above when this   particle try to jump to the right by one,
\item if $i$ is odd, $j=(i+1)/2$ and $Y_j^{i}(t^-)\ge 1$ then $Y_j^i(t)=Y_j^i(t^-)-1$,
\item if $i$ is even or $j\ne (i+1)/2$, and $Y_j^i(t^-)=Y_j^{i-1}(t^-)$ then particles don't move,
\item if $i$ is even or $j\ne (i+1)/2$, and $Y_j^i(t^-)>Y_j^{i-1}(t^-)$ then particles $(i,j),(i+1,j+1),\dots,(i+l,j+l)$ jump  to the left  by one for $l$  the largest integer such that $Y_{j+l}^{i+l}(t^-)=Y_j^i(t^-).$ Thus $$Y_j^{i}(t)=Y_j^{i}(t^-)-1,\dots,Y_{j+l}^{i+l}(t)=Y_{j+l}^{i+l}(t^-)-1.$$ 
\end{enumerate}

This random particle model is equivalent to a random tiling model with a wall, as it is explained in detail in \cite{BorodinKuan}.

  \section{Markov Kernel on the set of irreducible representations of the orthogonal group}\label{representation}
When a finite dimensionnal representation $V$ of a group $G$ is completely reducible, there is a natural way that we'll recall later in our particular case to associate to this decomposition a probability measure on the set of irreducible representations of $G$. The transition probabilities of the random process $(X^{k}(t),t\ge 0)$ which will be proved to be Markovian  are obtained in that manner. Actually we recover them considering decomposition into irreducible components of tensor products of particular  irreducible representations of the special orthogonal group.

Let $d$ be an integer greater than $2$. Let us recall some usual properties of  the finite dimensional representations of the compact group $SO(d)$ of   $d\times d$ orthogonal matrices with determinant equal to $1$ (see for instance \cite{Knapp} for more details).  The set of finite dimensional representations of $SO(d)$ is indexed by the set  
$$\{\lambda\in \R^r:2\lambda_r\in \N, \lambda_i-\lambda_{i+1}\in \N, i=1,\dots,r-1\},$$
when $d=2r+1$ and by the set
$$\{\lambda\in \R^r:\lambda_{r-1}+\lambda_r \in \N,\,  \lambda_i- \lambda_{i+1}\in \N,\, i=1,\dots,r-1\},$$
when $d=2r$. Actually we are only interested with representations indexed by a subset $\mathcal{W}_d$ of these sets defined by 
$$\mathcal{W}_d=\{\lambda\in \R^r:\lambda_r\in \N, \lambda_i-\lambda_{i+1}\in \N, \,i=1,\dots,r-1\},$$
when $d=2r+1$ and
$$\mathcal{W}_d=\{\lambda\in \R^r:\lambda_r\in \Z, \lambda_{r-1}+\lambda_r\in \N,\,  \lambda_i- \lambda_{i+1}\in \N,\, i=1,\dots,r-1\},$$
when $d=2r$.   For $\lambda\in \mathcal{W}_d$, using standard notations,  we denote by $V_\lambda$ the so called irreducible representation with highest weight $\lambda$ of $SO(d)$.  The subset of $\mathcal{W}_d$ of elements having non-negative components is denoted by $ \mathcal{W}^+_d$.

  Let $m$ be an integer and $\lambda$ an  element of  $\mathcal{W}_d$. Consider the irreducible representations $V_\lambda$ and $V_{\gamma_m}$ of  $SO(d)$, with $\gamma_m=(m,0,\cdots,0)$. The decomposition of the tensor product $V_\lambda\otimes V_{\gamma_m}$ into irreducible components is given by a Pieri-type formula for the orthogonal group. As it has been explained in \cite{Defosseux}, it can be deduced from \cite{Nakashima}.  We have 
    \begin{align}\label{MultB} V_{\lambda}\otimes V_{\gamma_m}=\oplus_{\beta}M_{\lambda,\gamma_m}(\beta) V_{\beta},
  \end{align}
  where the direct sum is over all  $\beta\in \mathcal{W}_d$ such that  
  \begin{itemize}
  \item when $d=2r+1$,   there exists an integer $s\in \{0,1\}$ and $c\in \mathbb{N}^{r}$ which satisfy $$\left\{
    \begin{array}{l}
      c\preceq \lambda , \quad c\preceq \beta \\
        \\
      \sum_{i=1}^{r}(\lambda_i-c_i+ \beta_i-c_i)+s=m,
    \end{array}
\right. $$ $s$ being equal to $0$ if $c_r=0$. In addition, the  multiplicity $M_{\lambda,\gamma_m}(\beta)$ of the irreducible representation with highest weight $\beta $ is the   number of $(c,s)\in\mathbb{N}^{r}\times \{0,1\}$ satisfying these relations.  
\item when $d=2r$, there exists  $c\in \mathbb{N}^{r-1}$ which verifies $$\left\{
    \begin{array}{l}
      c\preceq\vert \lambda \vert,\quad c\preceq \vert\beta \vert \\
        \\
     \sum_{k=1}^{r-1}(\lambda_k-c_k+ \beta_k-c_k)+ \vert\lambda_r-\mu_r \vert=m.
    \end{array}
\right.$$ In addition, the  multiplicity $M_{\lambda,\gamma_m}(\beta)$ of the irreducible representation with highest weight $\beta $ is the   number of $c\in\mathbb{N}^{r-1}$ satisfying these relations.  
\end{itemize}
Let us consider a family $(\mu_m)_{m\ge 0}$ of Markov kernels on $\mathcal{W}_d$ defined by 
$$\mu_m(\lambda,\beta)=\frac{\dim(V_\beta)}{\dim(V_\lambda)\dim(V_{\gamma_m})} M_{\lambda,\gamma_m}(\beta),$$
for $\lambda,\beta\in \mathcal{W}_d$ and $m\ge 0$. It is known that for $\lambda\in \mathcal{W}_d$ the dimension of $V_\lambda$ is given by the number $s_{d-1}(\lambda)$ defined in definition \ref{defnGT}.  Thus
$$\mu_m(\lambda,\beta)=\frac{s_{d-1}(\beta)}{s_{d-1}(\lambda)s_{d-1}(\gamma_m)} M_{\lambda,\gamma_m}(\beta).$$
Let $\xi_1,\dots ,\xi_{d}$ be independent geometric random variables with parameter $q$ and $\epsilon$ a Bernoulli random variable such that
$$\P(\epsilon =1)=1-\P(\epsilon =0)=\frac{q}{1+q}.$$
Consider a random variable $T$ on $\N$  defined by

$$T=\sum_{i=1}^{d-1}\xi_i+\epsilon,$$
when $d=2r+1$ and 
$$T=\vert\xi_1-\xi_2\vert+\sum_{i=3}^{d}\xi_i,$$
when $d=2r$.
\begin{lem} \label{nu} The law of $T$ is a measure $\nu$ on $\N$ defined by
$$\nu(m)=\frac{1}{1+q}(1-q)^{d-1}q^ms_{d-1}(\gamma_m),\quad m\in \N.$$
\end{lem} 
\begin{proof}
When $d=2r+1$, for $m=0$ the property is true. For $m\ge 1$
\begin{align*}
\P(T=m)&=\frac{q}{1+q}\P(\sum_{i=1}^{d-1}\xi_i=m-1)+\frac{1}{1+q}\P(\sum_{i=1}^{d-1}\xi_i=m)\\
&=\frac{1}{1+q}(1-q)^{d-1}q^m \Card\{(k_1,\dots,k_{d-1})\in \N^{d-1}: \sum_{i=1}^{d-1}k_i\in \{m-1,m\}\}\\
&=\frac{1}{1+q}(1-q)^{d-1}q^m\sum_{(k_1,\dots,k_{d-1})\in \N^{d-1}: \sum_{i=1}^{d-1}k_i=m}(21_{k_1\ge 1}+1_{k_1=0})\\
&=\frac{1}{1+q}(1-q)^{d-1}q^m s_{d-1}(\gamma_m).
\end{align*}
So the lemma is proved in the odd case. Moreover
$$\P(\vert \xi_1-\xi_2\vert=k )=\left\{
    \begin{array}{ll}
       2 \frac{1-q}{1+q}q^{k} & \mbox{if } k\ge 1, \\
        \\
        \frac{1-q}{1+q} & \mbox{otherwise.}
    \end{array}
\right.$$
Thus when $d=2r$, 
 \begin{align*}
\P(T=m)
&=\frac{1}{1+q}(1-q)^{d-1}q^m \sum_{(k_1,\dots,k_{d-1})\in \N^{d-1}: \sum_{i=1}^{d-1}k_i=m}(21_{k_1\ge 1}+1_{k_1=0})\\
&=\frac{1}{1+q}(1-q)^{d-1}q^m s_{d-1}(\gamma_m).
\end{align*} 
\end{proof}
Lemma \ref{nu} implies in particular that the measure $\nu$ is a probability measure. Thus one defines a Markov kernel $P_d$ on $\mathcal{W}_d$ by letting 
\begin{align}\label{Pd} 
P_d(\lambda,\beta)=\sum_{m=0}^{+\infty}\mu_m(\lambda,\beta)\nu(m),
\end{align}
for $\lambda,\beta\in \mathcal{W}_d$.  We'll  see that the kernel $P_{d}$ describes the evolution of the $(d-1)^{\textrm{th}}$ row of the random process on the set of Gelfand-Tsetlin patterns observed at integer times. 

 \begin{prop} \label{explicitPd} For $\lambda,\beta\in \mathcal{W}_d$, 
 \[
P_{d}(\lambda,\beta)=\sum_{c\in \N^r: c\preceq \lambda,\beta}(1-q)^{d-1}\frac{s_{d-1}(\beta)}{s_{d-1}(\lambda)}q^{\sum_{i=1}^{r}(\lambda_i+\beta_i-2c_i)}(1_{c_r>0}+\frac{1_{c_r=0}}{1+q})
\]
when $d=2r+1$ and
\[
P_{d}(\lambda, \beta)=\sum_{c\in \N^{r-1}: c\preceq \vert \lambda \vert, \vert\beta \vert}\frac{(1-q)^{d-1}}{q+1}\frac{s_{d-1}(\beta)}{s_{d-1}(\lambda)}q^{\sum_{i=1}^{r-1}(\lambda_i+ \beta_i-2c_i)+\vert \lambda_r-\beta_r\vert} 
\]
when $d=2r$.
 \end{prop}
 \begin{proof} The proposition follows immediately from the tensor product rules  recalled  for the decomposition (\ref{MultB}).
 \end{proof}
 \section{Random matrices}\label{matrices}
  Let us denote by ${\mathcal M}_{d,d'}$ the set of $d \times d'$ real  matrices.  A standard Gaussian variable on  ${{\mathcal M}_{d,d'}}$ is a random variable having  a density with respect to the Lebesgue measure on ${\mathcal M}_{d,d'}$ equal to
  $$M\in {\mathcal M}_{d,d'} \mapsto\frac{1}{\sqrt[dd']{2\pi}} \exp(-\frac{1}{2}\mbox{tr}(MM^* )).$$
   We write ${\mathcal A}_{d}$ for the set $\{M\in \mathcal{M}_{d,d}: M+M^*=0\}$ of antisymmetric $d\times d$ real matrices, and  $i{\mathcal A}_{d}$  for the set $\{iM: M\in \mathcal{A}_d\}$. Since a matrix in $i\mathcal{A}_d$ is Hermitian, it has real eigenvalues $\lambda_1 \geq \lambda_2 \geq \cdots \geq \lambda_d$. Morever, antisymmetry implies that  $\lambda_{d-i+1}=-\lambda_{i}$, for $i=1,\cdots, [d/2]+1$, in particular $\lambda_{[d/2]+1}=0$ when $d$ is odd. Consider the subset $\mathcal{C}_{d}$ of $\R_+^{[\frac{d}{2}]}$ defined by 
$$\mathcal{C}_d=\{x\in \R^{[\frac{d}{2}]}: x_1>\dots> x_{[\frac{d}{2}]}>0\},$$
and its closure $$\bar{{\mathcal{C}}}_d =\{x\in \R^{[\frac{d}{2}]}: x_1\ge \dots\ge x_{[\frac{d}{2}]}\ge0\}.$$
 \begin{defn} \label{asymptoticdim}  We define the function  $h_d$ on $\mathcal{C}_d$ by $$h_d(\lambda)=c_d(\lambda)^{-1}V_d(\lambda), \quad \lambda\in \mathcal{C}_d,$$ where the functions $V_d$ and $c_d$ are given by :
 \begin{align*}
 V_{n}(\lambda)&= \prod_{\substack{1\le i<j\le [\frac{d}{2}]}}(\lambda_i-\lambda_j)\prod_{\substack{1\le i<j\le [\frac{d}{2}]}}(\lambda_i+\lambda_j)\prod_{ \substack{1\le i\le [\frac{d}{2}]}} \lambda_i^{\varepsilon} ,\\  c_n(\lambda)&=\prod_{\substack{1\le i<j\le [\frac{d}{2}]}}(j-i)\prod_{\substack{1\le i<j\le [\frac{d}{2}]}}(d-j-i)\prod_{ \substack{1\le i\le [\frac{d}{2}]}}([\frac{d}{2}]+\frac{1}{2}-i)^{\varepsilon},
\end{align*}
whit $\varepsilon$ equal to $1$ when $d\notin2\N$ and $0$ otherwise.
\end{defn}
 The next proposition is a consequence of Propositions 4.8 and 5.1 of \cite{Defosseux}.
\begin{prop}\label{eigenvalues} Let  $(M(n),n\ge 0)$, be a random process  on $i\mathcal{A}_{d}$ defined by $$M(n)=\sum_{l=1}^{n}Y_l \begin{pmatrix} 0 &  i  \\
-i &  0       
\end{pmatrix}Y_l^*,$$ where the $Y_l$'s are independent standard Gaussian variables on $\mathcal{M}_{d,2}$. If $\Lambda(n)$ is the vector of $\bar{{\mathcal{C}}}_d$ whose components are the  $[\frac{d}{2}]$ largest eigenvalues of $M(n)$, $n\in \N$,  then the random process $(\Lambda(n),n\ge 0)$ is a Markov chain on $\bar{{\mathcal{C}}}_d$ with transition probabilities 
 \begin{align*}
p_d(x,dy)= \frac{h_d(y)}{h_d(x)} m_d(x,y)\,dy,\end{align*}
for $x,y\in \mathcal{C}_d$, where   $dy$ is the Lebesgue measure on $\R_+^{[\frac{d}{2}]}$ and 
$$m_d(x,y)=\int_{\mathbb{R}_+^{r}} 1_{\{z\preceq x,y\}}e^{-\sum_{i=1}^{m}(y_i+x_i-2z_i)}\,dz$$
when $d=2r+1$ and 
$$m_d(x,y)=\int_{\mathbb{R}_{+}^{r-1}} 1_{\{z\preceq \vert x\vert,\vert y\vert \} } e^{-\sum_{i=1}^{r-1}(x_i+y_i-2z_i)}(e^{-\vert x_r-y_{r}\vert  }+e^{-(x_r+y_r)})\, dz$$
when $d=2r$. 
 \end{prop}

\section{Results}\label{results}
The main result of our paper states in particular that if only  one row of the patterns $(X(t),t\ge 0)$ is considered by itself, it found to be a Markov process too. Actually we state the result for the process observed at integer times, even if the process observed at the whole time is also Markovian as we'll see in section \ref{proofs}.
\begin{theo}\label{maintheo}
The random process $(X^{k}(n))_{n\ge 0}$ is a Markov process on $\mathcal{W}^+_{k+1}$. If we denote by $R_k$ its transition kernel then
\begin{itemize} 
\item $R_1=R$. 
\item when $k$ is even  $R_k=P_{k+1}$,
\item when $k$ is an odd integer greater  than 2
\[
R_k(x,y)=\left\{
    \begin{array}{ll}
       P_{k+1}(x,y)+  P_{k+1}(x,\tilde{y}) & \mbox{if } y_{\frac{k+1}{2}}\ne 0 \\
        \\
       P_{k+1}(x,y) & \mbox{otherwise,}
    \end{array}
\right.
\]
\end{itemize}
for $x,y\in \mathcal{W}^+_{k+1}$, where $\tilde{y}=(y_1,\dots,y_{\frac{k-1}{2}},-y_{\frac{k+1}{2}})$.
\end{theo}
\begin{cor} \label{theoY}
Let  $(e_1,\dots,e_{[\frac{k+1}{2}]})$ be  the  canonical basis of $\R^{[\frac{k+1}{2}]}$. The random process $(Y^k(t),t\ge 0)$ is a Markov process with infinitesimal generator defined by
$$A_k(\lambda,\beta)=\left\{
    \begin{array}{ll}
       2 \frac{s_{k}(\beta)}{s_k(\lambda)}\,1_{\beta\in \mathcal{W}_k} & \mbox{if $k$ is odd, } \lambda_{\frac{k+1}{2}}= 0  \mbox{ and $\beta_{\frac{k+1}{2}}=1$} 
        \\
        \\
     \,\,\, \frac{s_{k}(\beta)}{s_k(\lambda)}\, 1_{\beta\in \mathcal{W}_k} & \mbox{otherwise,}
    \end{array}
\right. 
$$ 
for $\lambda\in \mathcal{W}_k,$ and $\beta\in \{\lambda +e_1,\dots, \lambda +e_{[\frac{k+1}{2}]}, \lambda-e_1,\dots, \lambda-e_{[\frac{k+1}{2}]}\}$. 
\end{cor}

If $(\Lambda(n),n\ge 0)$ is the process of eigenvalues considered in Proposition \ref{eigenvalues} with $d=k+1$ then the following theorem holds.
\begin{theo} \label{theoMat} Letting $q=1-\frac{1}{N}$, the process $(\frac{X^k(n)}{N}, n\ge 1)$ converges in distribution towards the process of eigenvalues $(\Lambda(n), n\ge 1)$ as $N$ goes to infinity.  
\end{theo}

\section{proofs}\label{proofs}
\subsection{Proof of Theorem \ref{maintheo}} For $k=1$, Theorem \ref{maintheo} is clearly true. The proof of the  theorem for $k\ge 2$ rests on an intertwining property and an application of a Pitman and Rogers  criterion given in \cite{PitmanRogers}.
\begin{Notation} Let $\xi_1$ and $\xi_2$ be two independent geometric random variables. For $x,a\in \N$ such that $x\ge a$, the law of the random variable
$$\max(a,x-\xi_1),$$
is denoted by $\overset{a\leftarrow}{P}(x,.)$.
For $x,b\in \N$ such that $x\le b$ we denote by  $\overset{\rightarrow b}{P}(x,.)$ and $\overset{\rightarrow b}{R}(x,.)$ the laws of the random variables
$$\min(b,x+\xi_1)\,\textrm{ and } \,\min(b,\vert x+\xi_1-\xi_2\vert).$$
For $x,y\in \R^2$ such that $x\le y$ we let 
$$P(x,y)=(1-q)q^{y-x}.$$
\end{Notation}
The two following lemmas are proved by straightforward computations.
\begin{lem}For $a,x,y\in \N$ such that $a\le y\le x$
\begin{align*}
\overset{a\leftarrow}{P}(x,y)&=\left\{
    \begin{array}{ll}
        (1-q)q^{x-y} & \mbox{if } a+1\le y \\
        q^{x-a} & \mbox{if } y=a.
    \end{array}
\right. 
\end{align*} 
For $b,x,y\in \N$ such that $b\ge y\ge x$
\begin{align*}
\overset{\rightarrow b}{P}(x,y)&=\left\{
    \begin{array}{ll}
        (1-q)q^{y-x} & \mbox{if } y\le b-1 \\
        q^{b-x} & \mbox{if } y=b.
    \end{array}
\right. 
\end{align*}
For $b,x,y\in \N$ such that $b\ge y, x$
\begin{align*}
\overset{\rightarrow b}{R}(x,y)&=\left\{
    \begin{array}{ll}
      \frac{1-q}{1+q}(q^{\vert y-x\vert }+q^{x+y}) & \mbox{if }  y\le b-1, y> 0 \\
      \\
       \frac{1-q}{1+q}q^{x }& \mbox{if }  y\le b-1, y= 0   \\
       \\
          \frac{1}{1+q}q^b(q^{-x }+q^{x}) & \mbox{if }  y=b, y> 0 \\
          \\
         1 & \mbox{if }  y=b, y=0.
    \end{array}
\right.
\end{align*}
\end{lem}
 
\begin{lem}\label{desintegration}
For $(x,y,z)\in \N^3$ such that $0<z\le y$
\begin{align}\label{desintegration1}
\sum_{u=0}^{z}(1_{u=0}+2\,1_{u>0})R(u,x)\overset{u\leftarrow}{P}(y,z)=  (1-q)(1_{x=0}+2\,1_{x>0})q^{x\vee z+y-2z}.
\end{align}
For $(x,y,a)\in \N^3$ such that $a\le y$ and $y\le x$
\begin{align}\label{desintegration2}
\sum_{u=a}^{y}q^u\overset{u\leftarrow}{P}(x,y)= q^{x-y}q^a.
\end{align}
For $(x,y,a)\in \N^3$ such that $y\le a$ and $x\le y$
\begin{align}\label{desintegration3}
\sum_{v=y}^{a}q^{-v}\overset{\rightarrow v}{P}(x,y)= q^{y-x}q^{-a}.
\end{align}
For $y\in \N,y'\in \N^*$ such that $y'\le a$
\begin{align}\label{desintegration4}
\sum_{v=y'}^{a}q^{v\vee y-2v}\overset{\rightarrow v}{R}(y\wedge v,y')= \frac{1}{1-q}q^{-a}R(y,y').
\end{align}
\end{lem}  
\bigskip
 Let us first prove Theorem \ref{maintheo}  for $k = 2$. Consider the set 
$$\mathcal{W}^+_{2,3}=\{(z,y)\in \N^2: z\le y\}.$$ Define a Markov kernel $S_2$ on  $\mathcal{W}^+_{2,3}$ by letting  
$$S_2((z_0,y_0),(z,y))=\left\{
    \begin{array}{ll}
      (1-q)^2\frac{s_2(y)}{s_2(y_0)} q^{y_0+y-2z}1_{z\le y_0\wedge y}& \mbox{if }  z> 0 \\
      \\
      \frac{(1-q)^2}{1+q}\frac{s_2(y)}{s_2(y_0)}q^{y_0+y} & \mbox{if }  z= 0.
    \end{array}
\right.
$$
for $(z_0,y_0),(z,y)\in \mathcal{W}^+_{2,3}$ and another one $L_2$ from $\mathcal{W}^+_{2,3}$ to $\N\times \mathcal{W}^+_{2,3}$ by letting 
$$L_2((z_0,y_0),(x,z,y))=(1_{x=0}+2\,1_{x>0})\frac{1}{s_2(y)}1_{x\le y}1_{(z_0,y_0)=(z,y)},$$
for $(z_0,y_0),(z,y)\in \mathcal{W}^+_{2,3}$ and $x\in \N$.   The fact that $S_2$ is a Markov kernel follows from Proposition \ref{explicitPd} with $d=3$. The random process 
$$(X_1^1(n),X_1^2(n-\frac{1}{2}),X_1^2(n))_{n\ge 1},$$
is clearly Markovian. Let us denote by $Q_2$ its transition kernel. Then $Q_2$, $L_2$ and $S_2$ satisfy the following intertwining property.
 \begin{lem}\label{intertwining2}
$$L_2Q_2=S_2L_2.$$
 \end{lem}
 \begin{proof} For $(x,z,y),(x',z',y')\in \N\times \mathcal{W}^+_{2,3}$ such that $x\le y$ and $x'\le y'$ \[
Q_2((x,z,y),(x',z',y'))=R(x,x')\overset{x\leftarrow}{P}(y,z')P(x'\vee z',y').
\]Thus   
\begin{align*} 
L_2Q_2((z,y),(x',z',y'))=\sum_{x=0}^{z'}\frac{s_1(x)}{s_2(y)}R(x,x')\overset{x\leftarrow}{P}(y,z')P(x'\vee z', y')
\end{align*}
As $L_2$, $S_2$ and $Q_2$ are Markov kernels, it is sufficient to prove the identity for $z'>0$. In that case the identity (\ref{desintegration1}) of Lemma \ref{desintegration} implies that 
\[
\sum_{x=0}^{z'}\frac{(1_{x=0}+2\,1_{x>0})}{s_2(y)}R(x,x')\overset{x\leftarrow}{P}(y,z')=(1-q)(1_{x'=0}+2\,1_{x'>0}) q^{x'\vee z'+y-2z'}.
\]
Thus
\begin{align*} 
L_2Q_2((z,y),(x',z',y'))=  \frac{(1_{x'=0}+2\,1_{x'>0})}{s_2(y)}(1-q)^2q^{y+y'-2z'}, \end{align*} which proves that 
$$L_2Q_2=S_2L_2.$$  
\end{proof}
 Since the random process 
$$(X_1^1(n),X_1^2(n-\frac{1}{2}),X_1^2(n))_{n\ge 1}$$
is Markovian with transition kernel $Q_2$,  the intertwining property stated in Lemma \ref{intertwining2} and the criterion of Pitman and Rogers given in \cite{PitmanRogers} imply the following proposition. It states that the second row of the random process on the  set of Gelfand-Tsetlin patterns is Markovian and gives its transition kernel. 
\begin{prop} \label{bigproc1}
We let $X_1^{2}(-\frac{1}{2})=X_1^{2}(1)=0$.  The random process
$$(X_1^{2}(n-\frac{1}{2}),X_1^{2}(n))_{n\ge 0}$$
is a Markov process on $\mathcal{W}^+_{2,3}$ with transition kernel $S_2$.
\end{prop} 
As for $(z,y)\in \mathcal{W}^+_{2,3}$ the probability $S_2((z,y), .)$ doesn't depend on $z$, Theorem \ref{maintheo} easily follows  from 
Proposition \ref{bigproc1} when $k = 2$. 
\\
\paragraph{} For the general case one defines the random process $(Z^k(n),Y^k(n))_{n\ge 1}$ by letting
\begin{align*}
Z^{k}(n)&=(X_1^{k}(n-\frac{1}{2}),\dots,X_{[\frac{k}{2}]}^{k}(n-\frac{1}{2})), \\
Y^{k}(n)&=X^{k}(n),
\end{align*}
for $n\ge 1$ and $
Z^{k}(0)=Y^{k}(0)=0$. Let us notice that $Z^k$ is equal to $X^k$ when $k$ is even, whereas it is obtained from $X^k$ by deleting its smallest component  when $k$ is odd.  We consider the subset $\mathcal{W}^+_{k,k+1}$  of $\mathcal{W}^+_{k}\times \mathcal{W}^+_{k+1}$ defined by
$$\mathcal{W}^+_{k,k+1}=\{(z,y)\in \mathcal{W}^+_{k}\times \mathcal{W}^+_{k+1}: z\preceq y\},$$ and  define  a Markov kernel $S_k$ on $\mathcal{W}^+_{k,k+1}$ by letting for every $(z,y),(z',y')\in \mathcal{W}^+_{k,k+1}$ 
\begin{align}\label{Skeven}
S_{k}((z,y),(z',y'))&=(1-q)^{k}\frac{s_{k}(y')}{s_{k}(y)}q^{\sum_{i=1}^{r}(y_i+y'_i-2z_i)}(1_{z_r>0}+\frac{1_{z_r=0}}{1+q})1_{ z'\preceq y,y'}
\end{align} 
when $k=2r$, and  
\begin{align}\label{Skodd}
S_{k}((z,y),(z',y'))&=(1-q)^{k-1}\frac{s_{k}(y')}{s_{k}(y)}R(y_r,y'_r)q^{\sum_{i=1}^{r-1}(y_i+y'_i-2z_i)}1_{ z'\preceq y,y'} 
\end{align}
when $k=2r-1$. The fact that for $(z,y)\in \mathcal{W}^+_{k,k+1}$ the measure $S_k((z,y),.)$  is a  probability measure is a consequence of Proposition \ref{explicitPd} with $d=k+1$.
 \begin{notaetoile}
Since for $(z,y)\in \mathcal{W}^+_{k,k+1}$ the probability $S_k((z,y), .)$ doesn't depend on $z$, it is denoted by $S_k(y,.)$ when there is no ambiguity.
 \end{notaetoile}
 
Even if the Markov kernel $P_{k+1}$ is relevant for our purpose, we'll prove Theorem \ref{maintheo} showing that the Markov kernel $S_k$ describes the evolution of the $k^{\textrm{th}}$ row of the process $(X(t),t\ge 0)$ observed at the whole time. We'll prove it by induction on $k$.

\begin{lem} \label{Qk} 
If the random process $$(Z^{k-1}(n),Y^{k-1}(n))_{n\ge 1},$$ is a Markov process  on $\mathcal{W}^+_{k-1,k}$ with transition kernel $S_{k-1}$ then the random process
$$(Y^{k-1}(n),Z^{k}(n),Y^{k}(n))_{n\ge 1}$$
is a Markov process on  the set
$$\{(x,(z,y))\in \mathcal{W}^+_k\times \mathcal{W}^+_{k,k+1}: x\preceq y\}.$$
If we denote its transition kernel by $Q_{k}$ then for $(u,z,y),(x,z',y')\in  \mathcal{W}^+_k\times \mathcal{W}^+_{k,k+1}$ such that $u\preceq y$ and $x\preceq y'$  
\begin{align}\label{Qlodd} 
Q_{k}((u,z,y),(x,z',y'))=\sum_{v\in \N^{r-1}}S_{k-1}&(u,(v,x))\overset{\rightarrow v_{r-1}}{R}(y_r\wedge v_{r-1},y'_r) \nonumber\\
&\times \prod_{i=1}^{r-1}\overset{u_{i}\leftarrow}{P}(y_{i}\wedge v_{i-1},z'_{i})\prod_{i=1}^{r-1}\overset{ \rightarrow v_{i-1}}{P}(z'_{i}\vee x_{i},y'_{i}),
\end{align}
when $k=2r-1$ and
\begin{align}\label{Qleven}
Q_{k}((u,z,y),(x,z',y'))=\sum_{v\in \N^{r-1}}S_{k-1}&(u,(v,x))\overset{u_{r}\leftarrow}{P}(y_{r}\wedge v_{r-1},z'_{r}) \nonumber\\&\times \prod_{i=1}^{r-1} \overset{u_{i}\leftarrow}{P}(y_{i}\wedge v_{i-1},z'_{i})\prod_{i=1}^{r} \overset{ \rightarrow v_{i-1}}{P}(z'_{i}\vee x_{i},y'_{i}),
\end{align}
 when $k=2r$. In the odd and the even cases $v_0=+\infty$ and the sum runs over $v=(v_1,\dots,v_{r-1})\in \N^{r-1}$ such that $v_i\in\{y'_{i+1},\dots,x_i\wedge z'_i\}$, for $i\in\{1,\dots,r-1\}$
\end{lem}
\begin{proof} The dynamic of the model implies that the process $$(Z^{k-1}(n),Y^{k-1}(n),Z^{k}(n),Y^{k}(n),n\ge 0)$$ is Markovian. Since for $(z,y)\in \mathcal{W}^+_{k-1,k}$ the transition probability $S_{k-1}((z,y),.)$ doesn't depend on $z$, the Markovianity of the process  $$(Y^{k-1}(n),Z^{k}(n),Y^{k}(n),n\ge 0)$$ follows. Identities (\ref{Qlodd}) and (\ref{Qleven}) are deduced from the blocking and pushing interactions of the model.  
\end{proof}
Proof of Theorem \ref{maintheo} for every integer $k$  follows as in the case  when $k=2$ from an intertwining property. Let us define Markov Kernel $L_{k}$  from $ \mathcal{W}^+_{k,k+1}$ to $ \mathcal{W}^+_{k}\times  \mathcal{W}^+_{k,k+1}$  by letting for $x\in  \mathcal{W}^+_{k}$ and $(z,y),(z_0,y_0)\in  \mathcal{W}^+_{k,k+1}$
\begin{align}\label{Lkodd}
L_k((z_0,y_0),(x,y,z))=1_{(z_0,y_0)=(z,y)}\frac{s_{k-1}(x)}{s_k(y)}1_{x\preceq y},
\end{align}
when $k$ is odd and 
\begin{align}\label{Lkeven}
L_k((z_0,y_0),(x,y,z))=(1_{\{0\}}(x_{\frac{k}{2}})+2\,1_{\N^*}(x_{\frac{k}{2}}))1_{(z_0,y_0)=(z,y)}\frac{s_{k-1}(x)}{s_k(y)}1_{x\preceq y},
\end{align}
when $k$ is even. The following proposition generalizes Lemma \ref{intertwining2}.
\begin{prop} \label{intertwining}
The Markov kernels $S_k$, $L_k$ and $Q_k$ defined as in the identities (\ref{Skeven}), (\ref{Skodd}) and (\ref{Lkodd}),  (\ref{Lkeven})  and (\ref{Qlodd}), (\ref{Qleven}) satsify the intertwining
$$L_kQ_k=S_kL_k.$$
\end{prop}
\begin{proof} For $(z,y)\in \mathcal{W}^+_{k,k+1}$, $(x,z',y')\in \mathcal{W}^+_k\times \mathcal{W}^+_{k,k+1}$ such that $x\preceq y'$,
\begin{align*}
L_{k}Q_{k}((z,y),(x,z',y'))&=\sum_{u\in \mathcal{W}_{k}^+}L_{k}((z,y),(u,z,y))Q_{k}((u,z,y),(x,z',y')).\end{align*}
We prove separately the even and the odd cases.  When $k=2r$, the sum is equal to
\begin{align*}
&\sum_{(u,v)\in \N^r\times \N^{r-1} } \frac{s_{k-1}(x)}{s_k(y)} (1_{\{0\}}(u_r)+2\,1_{\N^*}(u_r))(1-q)^{2r-2} R(u_r,x_r)q^{\sum_{i=1}^{r-1}(x_i+u_i-2v_i)}\nonumber\\&\quad \quad \quad\quad \quad \quad\quad \quad\quad \quad\quad \times P(z'_{1}\vee x_{1},y'_{1})\prod_{i=1}^{r} \overset{u_{i}\leftarrow}{P}(y_{i}\wedge v_{i-1},z'_{i})\prod_{i=2}^{r}  \overset{ \rightarrow v_{i-1}}{P}(z'_{i}\vee x_{i},y'_{i}).
\end{align*}
where the sum runs over $(u,v)\in \N^r\times \N^{r-1}$ such that $u_r\in\{0,\dots,z_r'\}$,  $v_i\in\{y'_{i+1},\dots,x_i\wedge z'_i\}$, $u_i\in \{v_i\vee y_{i+1},\dots,z'_i\}$, for $i\in\{1,\dots,r-1\}$.
Thus the sum equals
\begin{align*}
&\sum_{v\in \N^{r-1}} \frac{s_{k-1}(x)}{s_k(y)} (1-q)^{2r-2}q^{\sum_{i=1}^{r-1}x_i}P(z'_{1}\vee x_{1},y'_{1}) \prod_{i=2}^{r}  q^{-2v_{i-1}}\overset{ \rightarrow v_{i-1}}{P}(z'_{i}\vee x_{i},y'_{i})\\&\quad \quad \quad\quad \quad  \times \sum_{ u\in \N^r }  (1_{\{0\}}(u_r)+2\,1_{\N^*}(u_r))(1-q)^{2r-2} R(u_r,x_r) \prod_{i=1}^{r} q^{u_i}\overset{u_{i}\leftarrow}{P}(y_{i}\wedge v_{i-1},z'_{i}) .
\end{align*}
For a fixed $v$ the sum over $u$ is equal to
$$\sum_{u_r=0}^{z_r'}(1_{\{0\}}(u_r)+2\,1_{\N^*}(u_r))R(u_r,x_r)\overset{u_{r}\leftarrow}{P}(y_{r}\wedge v_{r-1},z'_{r})  \prod_{i=1}^{r-1}\sum_{u_i=v_i\vee y_{i+1}}^{z_i'}q^{u_i}\overset{u_{i}\leftarrow}{P}(y_{i}\wedge v_{i-1},z'_{i}).$$
Since $L_k$ and $Q_k$ are Markov kernels it is sufficient to consider the case when $z_r>0$. In that case, identities (\ref{desintegration1}) and (\ref{desintegration2}) of Lemma  \ref{desintegration}  imply that the   sum over $u$ equals
$$(1_{\{0\}}(x_r)+2\,1_{\N^*}(x_r))q^{x_r\vee z'_r+y_r\wedge v_{r-1}-2z_{r}'}(1-q)\prod_{i=1}^{r-1}q^{y_{i}\wedge v_{i-1}-z_{i}'+v_i\vee y_{i+1}},$$
i.e.
$$(1_{\{0\}}(x_r)+2\,1_{\N^*}(x_r))q^{x_r\vee z'_r+y_r-2z'_r+\sum_{i=1}^{r-1}y_i+v_i-z_i'}(1-q).$$
Thus
\begin{align*}
L_{2r}Q_{2r}((z,y),(x,z',y'))
\end{align*}
equals 
\begin{align*}
& \frac{s_{k-1}(x)}{s_k(y)} (1-q)^{2r-1}(1_{\{0\}}(x_r)+2\,1_{\N^*}(x_r))q^{x_r\vee z'_r+y_r-2z'_r+\sum_{i=1}^{r-1}y_i-z_i'}q^{\sum_{i=1}^{r-1}x_i} \\&\quad \quad \quad \quad\quad \quad \quad \quad\quad \quad  \times P(z'_{1}\vee x_{1},y'_{1})\prod_{i=2}^{r}  \sum_{v_{i-1}=y_{i}}^{x_{i-1}\vee z'_{i-1}} q^{-v_{i-1}}\overset{ \rightarrow v_{i-1}}{P}(z'_{i}\vee x_{i},y'_{i}).
\end{align*}
Identity (\ref{desintegration3}) of Lemma \ref{desintegration} gives that 
\begin{align*}
 \prod_{i=2}^{r}  \sum_{v_{i-1}=y_{i}}^{x_{i-1}\vee z'_{i-1}} q^{-v_{i-1}}\overset{ \rightarrow v_{i-1}}{P}(z'_{i}\vee x_{i},y'_{i})&=\prod_{i=2}^{r}q^{y'_i-z'_i\vee x_i-x_{i-1}\wedge z_{i-1}'}\\
&=q^{y_r'-z_r'\vee x_r-x_1\wedge z'_1}q^{\sum_{i=2}^{r-1}y'_i-x_i-z_i'},
\end{align*}
which implies 
 \begin{align*}
L_{2r}Q_{2r}((z,y),(x,z',y'))= \frac{s_{k-1}(x)}{s_k(y)} (1-q)^{2r}(1_{\{0\}}(x_r)+2\,1_{\N^*}(x_r))q^{\sum_{i=1}^{r}y_i+y'_i-2z_i'},
\end{align*}
and achieves the proof for the even case. Similarly when $k=2r-1$ 
 \begin{align*}
L_{2r-1}Q_{2r-1}((z,y),(x,z',y'))& =\sum_{u,v\in \N^{r-1}}\frac{s_{k-1}(x)}{s_{k}(y)}q^{\sum_{i=1}^{r-1}x_i-2v_i}\overset{\rightarrow v_{r-1}}{R}(y_r\wedge v_{r-1},y'_r) \\&   \quad \quad \quad \times \prod_{i=1}^{r-1}q^{u_i}\overset{u_{i}\leftarrow}{P}(y_{i}\wedge v_{i-1},z'_{i})\prod_{i=1}^{r-1}\overset{ \rightarrow v_{i-1}}{P}(z'_{i}\vee x_{i},y'_{i}),
\end{align*}
where the sum runs over $(u,v)\in \N^{r-1}\times \N^{r-1}$ such that   $v_i\in\{y'_{i+1},\dots,x_i\wedge z'_i\}$, $u_i\in \{v_i\vee y_{i+1},\dots,z'_i\}$, for $i\in\{1,\dots,r-1\}$. 
We obtain the intertwining in a quite similar way as in the even case, using identities (\ref{desintegration2}), (\ref{desintegration3}) and (\ref{desintegration4}) of Lemma \ref{desintegration}.  \end{proof}
The following proposition generalizes Proposition \ref{bigproc1}.
\begin{prop}  \label{mainprop} The random process $(Z^k(n),Y^k(n))_{n\ge 1}$, 
is Markovian with transition kernel $S_k$ defined in (\ref{Skeven}) and (\ref{Skodd}). 
\end{prop} 
\begin{proof} The random process $(X^k(t),t\ge 0)$ is conditionally independent of the processes $(X^{l}(t),t\ge 0)$, for $l=1,\dots,k-2$, given the process $(X^{k-1}(t),t\ge 0)$. So the property can be proved by induction on $k$. Proposition \ref{bigproc1} claims that Proposition \ref{mainprop} is true for $k=2$. Suppose that the proposition is true for a fixed interger $k-1$ greater that $1$. Lemma \ref{Qk} implies that  the process
$$(Y^{k-1}(n),Z^{k}(n),Y^{k}(n))_{n\ge 1}$$
is Markovian with transition kernel $Q_k$. The intertwining property of Proposition \ref{intertwining} implies, by using the Pitman and Rogers criterion given in \cite{PitmanRogers}, that   the process $$(Z^{k}(n),Y^{k}(n))_{n\ge 1}$$ is Markovian with probability $S_k$. \end{proof} 
The Markov kernels $P_{k+1}$ and $S_k$ satisfy $$P_{k+1}(y,y')=\sum_{z'\in \mathcal{W}_k^+}S_k(y,(z',y')), \quad y,y'\in  \mathcal{W}_{k+1}^+.$$ 
Thus Theorem \ref{maintheo} is an immediate  corollary of Proposition \ref{mainprop}.

\subsection{Proof of Corollary \ref{theoY}} 

The proof of Corollary \ref{theoY} rests on a similar argument as in section 2.7 of \cite{BorodinFerrari}.
\begin{lem} \label{T1T2} Let $T_1(q)$ and $T_2(q)$ be two (possibly infinite) lower and upper triangular matrices, whose matrix coefficients are polynomials in an indeterminate $q>0$:
$$\left\{
    \begin{array}{l}
       T_1(q)=A_0+qA_1+q^2A_2+\dots, \\
       \\
        T_2(q)=B_0+qB_1+q^2B_2+\dots,
    \end{array}
\right. $$
and assume that $A_0=B_0=I$. Then for $t\in \R_+$,
$$\lim_{q\to 0}(T_1(q)T_2(q))^{[t/q]}=\exp(t(A_1+B_1)).$$
\end{lem}
\begin{proof} Because of  the triangularity assumption, the lemma follows, as in the proof of Lemma 2.21 of \cite{BorodinFerrari}, from the claim for finite size matrices, which is standard.
\end{proof}
Lemma \ref{T1T2} implies immediately the following proposition. 
\begin{prop} \label{conv0} Letting $q=\frac{1}{N}$, the process $(X([Nt]),t\ge 0)$ converges in distribution towards the process $(Y(t),t\ge 0)$ as $N$ goes to infinity.
\end{prop}
 \begin{proof} It follows from Lemma \ref{T1T2} by taking 
\[
\left\{
    \begin{array}{l}
       T_1(q)(x,y)=\P(X(n+\frac{1}{2})=y\vert X(n)=x), \\
       \\
        T_2(q)(x,y)=\P(X(n+1)=y\vert X(n+\frac{1}{2})=x),
    \end{array}
\right. 
\]
for $x,y\in GT_k$.
\end{proof} 
With the help of the identities of Proposition \ref{explicitPd}, Theorem \ref{maintheo} and Lemma 2.21 of \cite{BorodinFerrari} imply that the process $$(X^k([Nt]),t\ge 0)$$ converges towards a Markov process with infinitesimal  generator equal to $A_k$ as $N$ goes to infinity. The convergence stated in  Proposition \ref{conv0}  achieves the proof of Corollary  \ref{theoY}.

\subsection{Proof of Theorem \ref{theoMat}}
Let $(x_N)_{N\ge 1}$ be a sequence of elements of $\mathcal{W}^+_{k+1}$ such that $\frac{x_N}{N}$ converges to $x\in \mathcal{C}_{k+1}$ as $N$ goes to infinity and $(\nu_N)_{N\ge 1}$ be a sequence  of probability measures on $\mathcal{W}^+_{k+1}$  defined by 
$$\nu_N=\sum_{y\in \mathcal{W}^+_{k+1}}R_k(x_N,y)\delta_{\frac{1}{N}y}.$$
Propositions \ref{explicitPd} and \ref{eigenvalues} imply that the measure $\nu_N$ converges to the measure $p_{k+1}$ defined in Proposition \ref{eigenvalues} as $N$ goes to infinity. Theorem \ref{theoMat} follows.


\begin{thebibliography}{00}   
\bibitem{BorodinKuan}{A. Borodin and J. Kuan, {Random surface growth with a wall and Plancherel measures for $O(\infty)$}, Communications on pure and applied mathematics, 67 (2010), 831--894.}
\bibitem{BorodinFerrari}{A. Borodin and P. Ferrari, {Anisotropic growth of random 
surfaces in 2+1 dimensions},	arXiv:0804.3035v2 (2008) }
 \bibitem{Defosseux}{M. Defosseux, {Orbit measures and interlaced determinantal point processes}, Ann. Inst. H. Poincar\'e Probab. Statist.,  46 (2010) 209--249.} 
  \bibitem{DefosseuxLUE}{M. Defosseux,  {Generalized Laguerre Unitary Ensembles and an interacting particles model with a wall}, to appear in Electronic Communications in Probability.}
  \bibitem{Knapp}{{\sc Knapp, A.W.} (2002).  {\it Lie groups, beyond an introduction.} Second ed. 
Progress in mathematics, Vol. {\bf 140}. Birkh\"auser Boston Inc., Boston, MA. MR1920389 (2003c:22001)}
\bibitem{Nakashima}{{\sc Nakashima, T.} (1993). Crystal base and a generalization of the Littlewood-Richardson rule for the classical Lie algebras. {\it  Comm. Math. Phys. } {\bf 154} 
 215--243. MR1224078   (94f:17015) 
}
\bibitem{PitmanRogers}{J.W. Pitman and L.C.G. Rogers, {Markov functions}, Ann. Probab.,  9(4) (1981) 573--582.} 
  \end{thebibliography}
\end{document}